\documentclass[a4paper,12pt]{amsart}

\usepackage{amsthm,amssymb,amsmath,amscd}
\usepackage{amsthm}
\usepackage{pgf, tikz}
\usepackage{graphicx,subcaption}
\usepackage{verbatim}

\theoremstyle{definition}
\newtheorem{definition}{Definition}

\theoremstyle{plain}
\newtheorem{theorem}[definition]{Theorem}
\newtheorem{lemma}[definition]{Lemma}
\newtheorem{cor}[definition]{Corollary}

\newtheorem{claim}[definition]{Claim}
\newtheorem*{conj}
{Conjecture}

\theoremstyle{remark}

\newcommand{\lln}{\ln\ln n}
\newcommand{\llln}{\ln\ln\ln n}
\newcommand{\lllln}{\ln_{(4)}n}
\newcommand{\llllln}{\ln_{(5)}n}

\newcommand{\Z}{\mathbb{Z}}
\newcommand{\F}{\mathbb{F}}

\newcommand{\SH}{\mbox{SH}}

\begin{document}

\title
{Integral homology of random simplicial complexes }

\author{Tomasz \L{u}czak}

\address{Adam Mickiewicz University,
Faculty of Mathematics and Computer Science
ul.~Umultowska 87,
61-614 Pozna\'n, Poland}

\email{\tt tomasz@amu.edu.pl}

\author{Yuval Peled}

\address{School of Computer Science and Engineering
Edmond Safra Campus, Givat Ram
The Hebrew University
Jerusalem, 91904, Israel }

\email{\tt yuvalp@cs.huji.ac.il}

\thanks{T\L{} partially 
supported by NCN grant 2012/06/A/ST1/00261. YP is grateful to the Azrieli foundation for the award of an Azrieli fellowship.}



\date{\today}

\begin{abstract}
The random $2$-dimensional simplicial complex process starts with a complete graph on $n$ vertices, and in every step a new $2$-dimensional face, chosen uniformly at random, is added. We prove that with probability tending to $1$ as $n\to\infty$, the first homology group over $\Z$ vanishes at the very moment when all the edges are covered by triangular faces. 
\end{abstract}

\maketitle

\section{Introduction}
The topological study of random simplicial complexes was initiated by Linial and Meshulam ~\cite{LM}, who introduced the model $Y_d(n,p)$. This is a random simplicial complex on $n$ vertices that has full $(d-1)$-dimensional skeleton, and every $d$-dimensional face is included independently with probability $p$. The study of $Y_d(n,p)$ deals with its topological asymptotic properties when $n$ is large. Namely, for a given function $p=p(n)$, we 
say that $Y_d(n,p)$ has some property {\em asymptotically almost surely}
(abbreviated {\em a.a.s}) if the probability of having this property tends to $1$ as 
$n\to\infty$. A monotone property $P$ has a {\em sharp threshold} at $p=p(n)$ if there exists some negligible function $\varepsilon=\varepsilon(n)\to 0$ for which a.a.s $Y_d(n,(1+\varepsilon)p)$ has the property $P$, and $Y_d(n,(1-\varepsilon)p)$ does not.

Random simplicial complexes were introduced as high-dimensional analogs of the random graph model $G(n,p)$, which coincides with $Y_1(n,p)$. It is well known that the threshold probability for connectivity of $G=G(n,p)$ equals to $\frac{\ln n}{n}$, which is also the threshold for $G$ not having any isolated vertices (see, for instance, \cite{BB}). 
In fact, a stronger hitting-time result holds. Let us define the {\em random graph process} $\mathcal{G}=\{G(n,M):0\le M \le \binom n2 \}$, as a Markovian process which starts with $n$ isolated vertices, and a new uniformly random edge is added in every step. The hitting-time of some monotone property $P$, denoted by $h_n(P)$, is the smallest $M$ for which $G(n,M)$ has property $P$. Bollob{\'a}s and Thomason proved in \cite{BT} that a.a.s
\[
h_n(G\mbox{~is connected}) = h_n(\delta > 0),
\]
where $\delta$ denotes the minimal vertex degree in $G$.

The connectivity of random simplicial complexes was explored mainly by studying its $(d-1)$-dimensional homology. This can be motivated by the fact that graph connectivity is equivalent to the vanishing of its $0$-th homology. Linial and Meshulam showed ~\cite{LM} that the threshold probability of the vanishing of the first homology of $Y_2(n,p)$ over $\F_2$ is $\frac{2\ln n}{n}$. In analogy with the $1$-dimensional case, this is the threshold for $Y_2(n,p)$ not having any edge which is not covered by a triangular face.
This result was generalized by Meshulam and Wallach ~\cite{MW} to higher dimensions and arbitrary fixed abelian groups. Namely, they proved that the threshold probability of the vanishing of $H_{d-1}\left(Y_d(n,p),A\right)$ over any fixed abelian group $A$ is $\frac{d\ln n}{n}$. A similar result when $A$ is a field of characteristic zero was proved by  Hoffman, Kahle and Paquette using Garland's method ~\cite{HKP_GAR}.

In the $2$-dimensional case, it is natural to study the fundamental group of $Y_2(n,p)$. Babson, Hoffman and Kahle proved in \cite{BHK} that the threshold for simple-connectivity is  of order $n^{-\frac 12}$, substantially larger than that for homological connectivity. Tighter upper bounds were obtained in \cite{KPS,WG}

The Linial-Meshulam theorem was also generalized in the direction of a hitting-time result. 
Consider the random $2$-dimensional complex process $\mathcal{Y}_2=\{ Y_2(n,M):
0\le M\le  {\binom{n}{3}}\}$, defined as the Markov chain in which 
we start with a complete graph on $n$ vertices, and in every step $M$ a new $2$-dimensional face, chosen uniformly at random, is added. Kahle and Pittel~\cite{KP} showed that a.a.s
\[
h_n(H_1(Y;\F_2)=0) = h_n(\delta > 0).
\]
Here $\delta$ denotes the minimal edge degree in $Y$, i.e., the smallest number of triangular faces of $Y$ in which any edge is contained.

The problem of computing the threshold for the vanishing of the integral $(d-1)$-dimensional homology of $Y=Y_d(n,p)$ is not resolved by these results. The vanishing of $H_{d-1}(Y;\Z)$ requires the vanishing of $H_{d-1}(Y;A)$ for {\em every} abelian group, and the result of ~\cite{MW} applies only for groups with a fixed or slowly growing size. It is commonly believed that this threshold also coincides with the coverage of every $(d-1)$-dimensional face by a $d$-face. A breakthrough in this problem by Hoffman, Kahle and Paquette ~\cite{HKP} yielded an upper bound of $\frac{80d\ln n}{n}$ for the threshold probability. Several elements of our proof are inspired by their new approach.

Our main result is a hitting-time result for the vanishing of the integral homology of the random $2$-dimensional complex process $\mathcal{Y}_2=\{ Y_2(n,M):
0\le M\le  {\binom{n}{3}}\}$. 

\begin{theorem}\label{thm:main}
Let $n$ be an integer, and suppose that $Y=\mathcal Y_2$ is the random $2$-dimensional complex process over $n$ vertices. Then, a.a.s
\[
h_n(H_1(Y;\Z)=0) = h_n(\delta > 0).
\]
\end{theorem}

Clearly, $H_1(Y(n,M);\Z)\ne 0$ for every $M<h_n(\delta > 0)$, and we show that this is tight. In particular, the theorem implies that the threshold for the vanishing of the integral first homology of $Y_2(n,p)$ is $\frac{2\ln n}n$. The cases $d\ge 3$ remain unresolved.

The structure of the paper is the following. In the next section we derive Theorem~\ref{thm:main}  from a number of lemmata. 
The subsequent section contains the proof of the most involved combinatorial ingredient of our argument --  Lemma~\ref{l:comb}.
We conclude the paper with some remarks and comments.

\section{Proof of main result}

One apparent difficulty in proving the vanishing of the integral homology is that it is equivalent to the vanishing of the homlogies over $\F_p$ for every prime $p$ (See ~\cite{HAT}). The first step in our proof is to reduce the number of primes we need to consider to $\exp(O(n^2))$.

\begin{lemma}\label{l:top}
Let $n$ be an integer and suppose $Y$ is an $n$-vertex $2$-dimensional simplicial complex with a full $1$-dimensional skeleton. 

If $H_{1}(Y;\F_p)=0$ for every prime $p\le\sqrt{3}^{{{n-1}\choose 2}}$, then $H_{1}(Y;\Z)=0$.
\end{lemma}     

\begin{proof}
We observe that $H_{1}(Y;\Z)$ is finite, since otherwise the $1$-dimensional homology of $Y$ over any field is not trivial. 
 Consider an inclusion-minimal subcomplex $T\subseteq Y$ with:
 \begin{enumerate}
 \item[(i)] a full $1$-skeleton, 
 \item[(ii)] a finite $1$-dimensional homology over $\Z$.
 \end{enumerate}
 In other words, $T$ is a $\mathbb Q$-acylic complex. 
Such complexes  were considered by Kalai~\cite{kalai} who showed
that $|H_{1}(T;\Z)|\le \sqrt{3}^{{{n-1}\choose 2}}$. 
Moreover, clearly $|H_{1}(Y;\Z)|\le |H_{1}(T;\Z)|$ 
since the former is a quotient group of the latter. 

Now suppose that  $H_{1}(Y;\Z)\ne 0$, and let $\F_{p^k}$ be one of its summands, where $p$ is a prime and $k\ge 1$ an integer. By the Universal Coefficients Theorem, $H_{1}(Y;\F_p)\ne 0$, and in addition,
$$p\le |H_{1}(Y;\Z)| \le \sqrt{3}^{{{n-1}\choose 2}}\,,$$ 
which contradicts our assumption.
\end{proof}      

Let $\tilde Y:=Y(n,h_n(\delta>0))$. The second step of our argument is to prove that $H_1(\tilde Y;\F)$ is "almost vanishing" for every field $\F$. However, instead of measuring $H_1(\tilde Y;\F)$ by its dimension, we use the concept of {\em homological shadow} which plays an important role in the phase transition of random simplicial complexes \cite{LP}. Given an $n$-vertex $2$-dimensional simplicial complex $Y$ with a full $1$-skeleton and a field $\F$, the $\F$-shadow of $Y$ is
\[
\SH(Y;\F) = \left\{ f \in \binom{[n]}{3} : H_1(Y;\F) = H_1(Y\cup\{f\};\F) \right\}.
\]
Here and throughout the paper we assume that the underlying vertex-set is $[n]=\{1,...,n\}$, and $\binom{[n]}{3}$ denotes the family of triples of vertices.  Note that in contrary to \cite{LP}, in this context it is more convenient to include the triangles of $Y$ in the shadow. 

Clearly, $H_1(Y;\F)=0$ is equivalent to $\SH(Y;\F) = \binom{[n]}3$, and in the following lemma we show that with very high probability, the size of the $\F$-shadow increases quickly in the random complex process.

\begin{lemma}\label{l:red}
Suppose $Y=Y_2(n,M)$, where $M=\frac{\ln n}{n}\binom {n}{3}$, and $\F$ is a field. Then, 
\[
\Pr\left[  |\SH(Y;\F)| <\binom{n}{3}- \frac{n^3}{\lln} \right] < \exp(-n^2).
\]
\end{lemma}  

\begin{proof}
Consider the process in which $M$ random triangles were generated one by one to form the complex $Y$. If at the very end there are more than $\frac{{n}^{3}}{\lln}$ which are not in the $\F$-shadow, then in every step during the whole process the probability to select one of them as the new $2$-dimensional face was at least $\frac{5}{\lln}$. However, every time such a triple was chosen, the dimension of $H_1(Y;\F)$ decreased by one, and this could not have happened more than $\dim H_1(Y_2(n,0);\F)=\binom {n-1}2$ times. Hence, the probability in question is bounded 
from above by $\Pr\left[\mbox{Bin}(M, \frac{5}{\lln})\le \binom {n-1}{2}\right]$, which is clearly smaller than $\exp(-n^2)$.
\end{proof}

The ideas of Lemmata \ref{l:top} and \ref{l:red} appeared, in rather similar contexts, in \cite{HKP}. The novelty of our approach is in exploiting the following useful properties of all the $\F$-shadows.

Let $Y$ be an $n$-vertex $2$-dimensional simplicial complex with full $1$-skeleton and $\F$ a field. Let us call the triangles in $\SH(Y;\F)$ {\em good}, and the others {\em bad}. Then, the following statements are easily observable. 
\begin{enumerate}
\item [(I)] Every triangle that can be triangulated by good triangles is good. Equivalently, there is no triangulation of the $2$-dimensional sphere $\mathbb{S}^2$ in which exactly one of the triangles is bad.
\item [(II)] Every triangle of $Y$ is good.
\end{enumerate}
In addition, for $Y$ and $\F$ that satisfy $|\SH(Y;\F)| \ge \binom{n}{3}- \frac{n^3}{\lln}$ as in Lemma \ref{l:red}, there holds that
\begin{enumerate}
\item [(III)] There are at most $n^3/\lln$ bad triangles.
\end{enumerate}
This leads us to the following key definition.
\begin{definition}
A partition of $\binom{[n]}3$ to good and bad triangles is {\em $Y$-shady} if it satisfies conditions (I),(II) and (III) above.
\end{definition}
In case the underlying complex $Y$ is arbitrary, we will refer to a {\em shady} partition and then condition (II) is trivial. We say that a shady partition is {\em complete} if all the triangles are good.
The last ingredient of our proof is the following Lemma.

\begin{lemma}\label{l:comb}
Let $n$ be an integer, and $\tilde Y=Y_2(n,h_n(\delta > 0))$. Then, a.a.s every  $\tilde{Y}$-shady partition is complete.
\end{lemma}   

A somewhat lengthy  proof of Lemma~\ref{l:comb} we postpone
until the next section of the paper. 
We conclude this section with a proof of the main theorem.

\begin{proof}[Proof of Theorem~\ref{thm:main}] Clearly, $h_n(\delta > 0) \le h_n(H_1(Y;\Z)=0)$. Hence, it suffices to prove that $H_1(\tilde Y;\Z)=0$, where  $\tilde Y=Y_2(n,h_n(\delta > 0))$.

Let $Y=Y_2(n,M)$, where $M=\frac{\ln n}{n}\binom {n}{3}$. There is a natural coupling such that a.a.s $Y\subseteq \tilde Y$, since the probability that $M \ge h_n(\delta>0)$ is negligible. Hence, by taking the union bound over all primes $p\le\sqrt{3}^{\binom{n-1}{2}}$, Lemma \ref{l:red} implies that a.a.s
\[
\mbox{$|\SH(\tilde Y;\F_p)|\ge \binom n3 - \frac{n^3}{\lln}$ for every prime $p\le\sqrt{3}^{{{n-1}\choose 2}}$}.
\]
In particular, $\SH(\tilde Y;\F_p)$ are the good triangles in a $\tilde Y$-shady partition. Therefore, by Lemma \ref{l:comb}, $\SH(\tilde Y;\F_p)=\binom{[n]}3$. In consequence, $H_1(\tilde Y;\F_p)=0$ for every $p\le\sqrt{3}^{\binom{n-1}{2}}$, which concludes the proof by Lemma~\ref{l:top}.
\end{proof}

\section{Proof of Lemma~\ref{l:comb}}\label{s:comb}

This section is devoted to the proof of Lemma~\ref{l:comb}.
In fact, we deduce it from a somewhat stronger result stated for
the binomial model, where computations are a bit simpler. Here and below  $\ln_{(i)}n$ denotes the natural logarithm 
iterated $i$ times.

Given a shady partition of $\binom{[n]}3$, we recursively extend the notion of `badness' onto edges and vertices. We say that an edge, or a pair of vertices, is {\em bad} if it is contained in more than $\frac{n}{\ln_{(5)}n} $  bad triples and, 
similarly, a vertex is {\em bad} if it belongs to more than $\frac{n}{\ln_{(4)}n} $ bad edges. 
All faces which are not bad are {\em good}.

Additionally, we say that shady partition is {\em elementary} if its bad edges are vertex-disjoint.

\begin{lemma}\label{l:comb2}
Let $np=2\ln n-\lllln$, and $Y=Y_2(n,p)$. Then, a.a.s. every $Y$-shady partition is elementary.
\end{lemma}    

We first show that Lemma \ref{l:comb2} implies Lemma \ref{l:comb}.

\begin{claim}\label{fact2}
In every shady partition, a triple that contains three good edges is good.
\end{claim}
\begin{proof}
Let $f=xyz$ be such a triple. For every vertex $v\notin f$, the triangles $xyv$, $xzv$ and $yzv$ form a triangulation of $f$ (See Figure \ref{fig:trian1}). Since all edges contained in $f$ are good, each of them is  contained in at most $\frac n{\ln_{(5)}n} $
bad triples. In consequence, there exists a vertex $v$, in fact,  at least
$n-3-\frac {3n}{\ln_{(5)}n}$ of them, such that the triangles $xyv$, $xzv$ and $yzv$ are all good. Hence, $f$ is good.
\end{proof}

\begin{figure}
    \centering
    \begin{subfigure}[b]{0.48\linewidth}        
        \centering
    \begin{tikzpicture}
	\draw (-2,0)--(0,3)--(2,0)--(-2,0);
	\draw (0,1.3)--(-2,0);
	\draw (0,1.3)--(2,0);
	\draw (0,1.3)--(0,3);
	\draw (0,1) node {$v$};
	\draw (-2.2,0) node {$x$};
	\draw (2.2,0) node {$y$};
	\draw (0,3.2) node {$z$};
    \end{tikzpicture}
        \caption{}
        \label{fig:A}
    \end{subfigure}
    \begin{subfigure}[b]{0.48\linewidth}        
        \centering
           \begin{tikzpicture}
	\draw (-2,0)--(0,3)--(2,0)--(-2,0);
	\draw (0,.8)--(-2,0);
	\draw (0,.8)--(2,0);
	\draw (-0.1,1.8)--(2,0);
	\draw (-0.1,1.8)--(-2,0);
	\draw (-0.1,1.8)--(0,3);
	\draw (-0.1,1.8)--(0,.8);
	\draw (0,.6) node {$v$};
	\draw (0.15,1.85) node {$w$};
	\draw (-2.2,0) node {$x$};
	\draw (2.2,0) node {$y$};
	\draw (0,3.2) node {$z$};
    \end{tikzpicture}
        \caption{}
        \label{fig:B}
    \end{subfigure}
    \caption{The triangulations used in the proofs of (A) Claim \ref{fact2} and (B) Lemma \ref{l:comb2}.}
    \label{fig:trian1}
\end{figure}
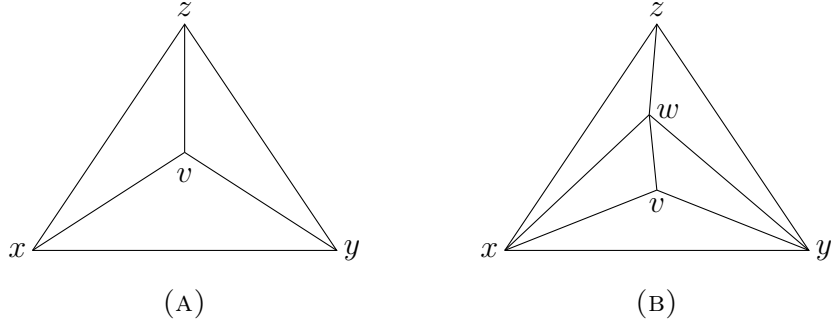

\begin{proof}[Proof of Lemma~\ref{l:comb}]
By a standard argument we couple $Y=Y_2(n,p)$ and $\tilde Y=Y_2(n,h_n(\delta>0))$ such that a.a.s $Y\subseteq\tilde Y$. In consequence, every $\tilde Y$-shady partition is elementary. The proof is concluded by showing that every elementary shady partition in which every edge is covered by a good triangle, is complete.

Suppose that there exists a bad triangle $xyz$. By Claim \ref{fact2} we assume, wlog, that the edge $xy$ is bad. By the assumption that every edge is covered, there exists a vertex $v$ such that $xyv$ is good. The edges $vx$, $xz$, $zy$ and $yv$ are good, because $xy$ is bad and the partition is elementary. Hence, there is a vertex $w$, in fact,  at least
$n-4-\frac {4n}{\ln_{(5)}n}$ of them, such that all the triangles $vxw$, $xzw$, $zyw$ and $yvw$ are good. However, together with $xyv$ they form a triangulation of $xyz$ by good triangles (See Figure \ref{fig:trian1}), in contradiction to $xyz$'s badness.
\end{proof}
\begin{figure}

\end{figure}

\subsection{Proof of Lemma \ref{l:comb2}}
Let $Y=Y_2(n,p)$, where $np=2\ln n-\lllln$. The proof of Lemma \ref{l:comb2} is done in three steps. Initially, in Corollary \ref{cor:very_good} we show that a.a.s in every $Y$-shady partition, there is a big subset $W$ of vertices, such that all the triangles of $W$ are good. Afterwards, for every vertex $v\notin W$, we find a large connected component in the subgraph of the link $\mbox{lk}_v(Y)$ induced by $W$, and deduce that all the vertices are good. Finally, we show that $Y$-shady partitions with only good vertices are elementary.

We start with some useful combinatorial observations about shady partitions.
\begin{claim}{\label{c:comb_shady}}
Let $C$ be a shady partition,
\begin{enumerate}
\item Denote by $b_i(C)$ the number of bad $i$-dimensional faces for $i=0,1$. Then,
\[
b_1(C) \le \frac{3n^2\ln_{(5)}{n}}{\lln}~~,~~b_0(C) \le \frac{6n\ln_{(5)}{n}\cdot \ln_{(4)}{n}}{\lln}\le \frac{n}{\lllln}.
\]
\\

\item Suppose that $x$ and $y$ are good vertices. Then, there exists at least $n-\frac{3n}{\lllln}$
good vertices $v$ such that the edges $vx$ and $vy$ are good.
\label{fact4}\\

\item Suppose that $x,y$ and $v$ are good vertices such that $vx$ and $vy$ are good edges. If $vxy$ is good, then $xy$ is good.
\label{fact5}\\
\end{enumerate}
\end{claim}
\begin{proof}
\begin{enumerate}
\item The first item follows inductively from the definitions.\\
\item A vertex $v$ fails to satisfy the requirements if (a) $v$ is bad, (b) $vx$ is bad, or (c) $vy$ is bad. Since $x,y$ are good, each such constraint eliminates at most $\frac{n}{\ln_{(4)}n}$ vertices.
\item Consider a vertex $z$ such that $zx,zy$ and $zv$ are good edges. By Claim \ref{fact2}, $xvz$ and $yvz$ are good, and $vxy$ is good by our assumption. In consequence, $xyz$ is also good. Since the number of these vertices $z$ is at least $n-\frac{3n}{\lllln} > n-\frac{n}{\llllln}$, $xy$ is a good edge.
\end{enumerate}
\end{proof}

\begin{lemma}\label{l:few}
A.a.s. in every $Y$-shady partition there are fewer than $n/\lln$ bad edges with two good endpoints.
\end{lemma}

\begin{proof}
We use the first moment method as following. 
We start by selecting the set of good vertices and $m$ bad edges on these good vertices, for $m\ge m_0=n/\lln$. Afterwards, we sample the complex $Y$ and bound the failure probability that the selected configuration can be induced by a $Y$-shady partition.
On the one hand,  our selection must be consistent with the second item of Claim \ref{c:comb_shady}, thus for every bad edge $xy$ there are at least  $n-\frac{3n}{\lllln}$ vertices $v$ such that both $xv$ and $yv$ are good. On the other hand, by the third item of the claim, the triangle $vxy$ must be bad for every such $xy$ and $v$. Therefore, our selection points at at least $mn(1-o(1))$ distinct triangles that $Y$ does not contain. This happens with probability at most 
$(1-p)^{mn(1-o(1))} \le \exp{(-2m\ln{n}(1-o(1))}$.

We bound the number  of choices of good vertices by $2^n$ and the number of choices of bad edges with good endpoints by $\binom{\binom{n}{2}}{m}$. The total failure probability is at most

\begin{equation*}\label{eq11}
\begin{aligned}
&2^n\sum_{m\ge m_0}\binom{\binom{n}{2}}{m}(1-p)^{mn(1-o(1))}\\
&\le \sum_{m\ge m_0} \exp{\left(n+m\left(1+2\ln{n}-\ln{m}-2\ln{n}(1-o(1))\right)\right)}\\
&\le \sum_{m\ge m_0} \exp{\left(n+m\left(-\ln{m}+o(\ln{n})\right)\right)}\\
&\le \exp\Big(n-0.5m_0\ln n\Big)
\\&\le \exp\Big(n\Big(1-\frac{\ln n}{\lln} \Big)\Big)\,,\\
\end{aligned}
\end{equation*}
and clearly tends to 0 as $n\to\infty$.
\end{proof}

\begin{cor}\label{cor:very_good}
A.a.s. in every $Y$-shady partition, there is a subset $W$ of at least $n-\frac{n\llln}{\lln}$ vertices such that all the triangles in $W$ are good.
\end{cor}	
\begin{proof}
By Claim \ref{fact2}, $W$ can be obtained by deleting every vertex that is contained in a bad edge from the set of good vertices. By Lemma \ref{l:few}, 
\[
|W|\ge n - b_0(C) - 2\frac{n}{\lln} \ge n-\frac{n\llln}{\lln}.
\]
\end{proof}

Note that the existence of $W$ from Corollary \ref{cor:very_good} does not exclude, in principle, that there is a vertex $v$ such that
all edges containing $v$ are bad. We shall show that it is not the case, and moreover, that all the vertices are good.
Recall that the link $\mbox{lk}_v(Y)$ of a vertex $v$ in the 2-complex $Y$ is a graph on $[n]\setminus\{v\}$ whose edge-set is $\{xy: xyv\in Y\}$.

\begin{claim}\label{l:link0}
A.a.s. for every vertex $v$, and every $W\subseteq [n]\setminus\{v\}$ of size at least $n-\frac{n\llln}{\lln}$, the subgraph of $\mbox{lk}_v(Y)$ induced by $W$  has a connected component of size at least $n-\frac{n}{\llln}.$
\end{claim}
\begin{proof}
Set $k_0=\frac{n}{2\llln}$, and fix $v$ and $W$ as above. The graph $\mbox{lk}_v(Y)$ is a $G(n-1,p)$ graph, hence the probability that the subgraph induced by $W$ contains a cut that separates a set of $k$ vertices, where $k_0\le k\le |W|/2$, is at most
\begin{equation*}\label{eq11}
\begin{aligned}
&\sum_{k=k_0}^{\frac {|W|}2} \binom nk (1-p)^{k(|W|-k)}\\
&\le\sum_{k=k_0}^{\frac n2} \binom nk (1-p)^{\frac{n^2}{8\llln}}\\
&\le\exp(-n\lln).
\end{aligned}
\end{equation*}
Hence, with probability at most $\exp(-n\lln)$ the subgraph contains a connected component of size at least $|W|-k_0\ge n-\frac{n}{\llln}$. The proof is concluded by taking the union bound over $v$ and $W$.
\end{proof}

\begin{lemma}\label{l:good}
A.a.s. in every $Y$-shady partition, all the vertices are good.
\end{lemma}
\begin{proof}
In order to show that a vertex $v$ is good, it suffices to prove that it is contained in at most $\frac{2n^2}{\llln}$ bad triangles. 
Indeed, since every bad edge is contained in at least $\frac{n}{\llllln}$ bad triangles, $v$ can be contained in at most $\frac{6n\llllln}{\llln}$ bad edges.\\

Let $C$ be a $Y$-shady partition, and $W$ be a set of at least $n-\frac{n\llln}{\lln}$ vertices such that all the triangles in $W$ are good. By Corollary \ref{cor:very_good}, such a  $W$ exists for every $C$ a.a.s. Clearly, every vertex in $W$ is contained in at least $\binom{|W|-1}{2}$ good triangles, so we only need to consider a vertex $v\notin W$. Let $H$ be a connected component of the subgraph of $\mbox{lk}_v(Y)$ induced by at least $n-\frac{n}{\llln}$ vertices of $W$. By Claim \ref{l:link0}, such component exists for every $v$ and $W$ a.a.s. We claim that for every two vertices $x,y \in H$, the triangle $vxy$ is good. Indeed, let $x=x_0,x_1,...,x_s=y$ be a simple path from $x$ to $y$ in $H$. The triangles $\{y,x_i,x_{i+1}\},~i=0,...,s-2$ are good because their vertices are in $W$. In addition, the triangles $\{v,x_i,x_{i+1}\},~i=0,...,s-1$ are good since they belong to $Y$. The claim is concluded by observing that the triangle $vxy$ can be triangulated by these good triangles (See Figure \ref{fig:tri2}). Therefore, every vertex is contained in at most $\binom n2 - \binom{n-\frac{n}{\llln}}2\le\frac{2n^2}{\llln}$ bad triangles.
\end{proof}

\begin{figure}
    \begin{tikzpicture}
	\draw (-2,0)--(0,3)--(4.5,0)--(-2,0);
	\draw (-2,0)--(-.9,.7)--(0,.9);
	\draw (-.9,.7)--(4.5,0)--(0,.9);
	\draw (0,3)--(-.9,.7)--(0,3)--(0,.9);
	\draw [dotted] (0,3)--(0.8,0.9)--(4.5,0);
	\draw[dashed](0,.9)--(.8,.9)--(1.7,.85);
	\draw  (0,3)--(1.7,.85)--(4.5,0); 
	\draw (-2.2,0) node {$x$};
	\draw (-.88,0.42) node {$x_1$};
	\draw (2.15,.9) node {$x_{s-1}$};
	\draw (0,0.7) node {$x_2$};

	\draw (4.7,0) node {$y$};
	\draw (0,3.2) node {$v$};
    \end{tikzpicture}

    \caption{The triangulation of $vxy$ described in the proof of Lemma \ref{l:good}.}
    \label{fig:tri2}
\end{figure}
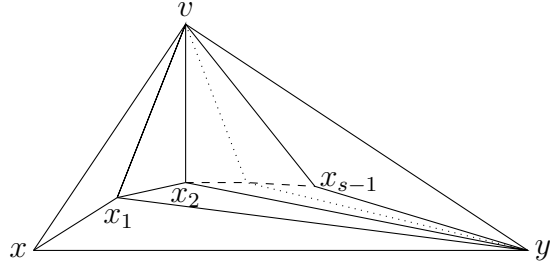

We turn to the proof of Lemma \ref{l:comb2}. The proof is very similar to Lemma \ref{l:few}, with one significant advantage: we do not need to choose the subset of good vertices, since all the vertices are good.

\begin{proof}[Proof of Lemma \ref{l:comb2}]
We condition on the event that in every $Y$-shady partition, all the vertices are good. We claim that a.a.s no graph $H$ with more than one edge could be a connected component of bad edges in a $Y$-shady partition. We do this by showing that the expected number of such graphs tends to $0$ as $n\to\infty.$

By Claim \ref{c:comb_shady}, for each bad edge $xy$, there are at least $n-\frac{3n}{\lllln}$ vertices $v$ such that (i) both $vx$ and $vy$ are good, and (ii) the triangle $vxy$ does not belong to $Y$. Note that these vertices $v$ can be read off $H$ since it is connected. In consequence, a graph $H$ with $m>1$ edges refutes our claim with probability at most
\[
(1-p)^{mn\left(1-\frac{3}{\lllln}\right)} \le \exp{\left(-2m\ln n\left( 1-\frac{4}{\lllln}\right)\right)}.
\]
We set $m_1:=n^{\frac{1}{\llllln}}$, and consider in separate the cases $m<m_1$ and $m\ge m_1$.
For $m\ge m_1$, we bound the number of graphs $H$ by $\binom{\binom n2}{m}$, thus the expectation in question is bounded by
\begin{equation*}\label{eq11}
\begin{aligned}
&\sum_{m\ge m_1}\binom{\binom{n}{2}}{m}\exp{\left(-2m\ln n\left( 1-\frac{4}{\lllln}\right)\right)}\\
&\le \sum_{m\ge m_1}\exp{\left(m\left(-\ln m + \frac{10\ln n}{\lllln}\right)\right)}\\
&\le \binom n2\exp{\left(-n^{\frac{1}{\llllln}}\left(\frac{\ln n}{\llllln} - \frac{10\ln n}{\lllln}\right)\right)},\\
\end{aligned}
\end{equation*}
which tends to $0$ as $n\to\infty.$

In case $m<m_1$, we bound the number of graphs $H$ as following. We select one edge of $H$ arbitrarily, and every additional edge is chosen such that one of its vertices has already been covered by previously selected edges. Hence, the number of these graphs is at most $n^2(2mn)^{m-1}$, and the expectation we are computing is at most

\begin{equation*}\label{eq11}
\begin{aligned}
&\sum_{m=2}^{m_1}n^2(2mn)^{m-1}\exp{\left(-2m\ln n\left( 1-\frac{4}{\lllln}\right)\right)}\\
&\le \sum_{m=2}^{m_1}n^{m+1}(2m)^{m-1}n^{-2m(1-o(1))}\\
&\le \sum_{m=2}^{m_1}\left(\frac{2m}{n^{1-o(1)}}\right)^{m-1}\\
&\le \frac{4m_1}{n^{1-o(1)}},\\
\end{aligned}
\end{equation*}
which tends to $0$ as $n\to\infty.$
\end{proof}

\section{Final remarks}\label{s:final}

The most natural open question that remains is determining, for $d>2$, the threshold for the vanishing of the integral $(d-1)$-dimensional homology of $Y_d(n,p)$ and obtaining a corresponding hitting-time result. It is possible that our argument can be modified to these cases, but it will surely become much longer and technically more involved.

It is not hard to use our argument to prove a stronger statement  than Theorem~\ref{thm:main}, which is similar to a result over $\F_2$ in \cite{KP}.
\begin{theorem}
Let $Y=Y_2(n,p)$, where $np=2\ln n-\lllln$. Then, a.a.s $H_1(Y_2(n,p);\Z)$ is a torsion-free group whose rank is equal to the number of uncovered edges.
\end{theorem}
We conjecture that torsion does not appear in much sparser random complexes. Recall from \cite {LP} that for every $d\ge 2$ there exists a real number $c_d^*$ such that $p=c_d^*/n$ is a one-sided sharp threshold for the non-triviality of $H_d(Y_d(n,p);\mathbb R)$.
\begin{conj}\label{conj1}  
For every $d\ge 2$ and $p=p(n)$ such that $|np-c_d^*|$ is bounded away from $0$, $H_{d-1}(Y_d(n,p);\Z)$ is torsion-free a.a.s.
\end{conj} 
Since the asymptotics of the real Betti numbers of $Y_d(n,p)$ are computed in \cite{LP}, this conjecture provides an almost complete understanding of the homology of random complexes. Nevertheless, numerical experiments suggest that in the random complex process a torsion subgroup of size $\exp(O(n^d))$ does appear at $M\approx \frac{c_d^*}{n}\binom{n}{d+1}$, shortly before the first appearance of a $d$-cycle  which is not a boundary of a $(d+1)$-simplex, and disappears immediately afterwards.
\vspace{0.7cm}

{\bf Acknowledgement.} This work was carried out when T{\L} visited the Institute for Mathematical Research (FIM) of ETH Z\"urich. He would like to thank FIM for the hospitality and for creating a stimulating research environment.


\end{document}